\documentclass[12pt,reqno]{amsart}
\usepackage{amssymb}
\usepackage{verbatim}
\usepackage{graphicx}
\usepackage{amsmath}
\usepackage{amsfonts,amssymb}

\usepackage[active]{srcltx}

\newtheorem{theorem}{Theorem}[section]
\newtheorem{lemma}[theorem]{Lemma}
\newtheorem{corollary}[theorem]{Corollary}
\newtheorem{proposition}[theorem]{Proposition}

\newtheorem*{thmM}{Main Theorem}
\newtheorem*{thmS}{Sharkovsky Theorem}

\theoremstyle{definition}
\newtheorem{definition}[theorem]{Definition}

\theoremstyle{remark}
\newtheorem{remark}[theorem]{Remark}

\numberwithin{equation}{section}

\newcommand{\N}{\mathcal{N}}
\newcommand{\M}{\mathcal{M}}
\newcommand{\K}{\mathcal{K}}

\newcommand{\sm}{\setminus}
\newcommand{\sh}{\mathrm{Sh}}
\newcommand{\dtp}{\mathrm{RT}}
\newcommand{\sha}{\succ\mkern-14mu_s\;}
\newcommand{\Sh}{\operatorname{Sh}}
\newcommand{\Per}{\operatorname{P}}
\newcommand{\Tow}{\operatorname{Tow}}

\renewcommand\le{\leqslant}
\renewcommand\ge{\geqslant}

\begin{document}

\date{November 23, 2018}

\title{Renormalization towers and their forcing}

\author[A.~Blokh]{Alexander~Blokh}

\author[M.~Misiurewicz]{Micha{\l}~Misiurewicz}

\address[Alexander~Blokh]
{Department of Mathematics\\ University of Alabama at Birmingham\\
Birmingham, AL 35294}
\email[Alexander~Blokh]{ablokh@math.uab.edu}

\address[Micha{\l}~Misiurewicz]
{Department of Mathematical Sciences\\Indiana University-Pur\-due
  University Indianapolis\\ 402 N. Blackford Street\\
Indianapolis, IN 46202}
\email[Micha{\l}~Misiurewicz]{mmisiure@math.iupui.edu}



\thanks{Research of Micha{\l} Misiurewicz was supported by grant
  number 426602 from the Simons Foundation.}

\begin{abstract}
A cyclic permutation $\pi:\{1, \dots, N\}\to \{1, \dots, N\}$ has a
\emph{block structure} if there is a partition of $\{1, \dots, N\}$
into $k\notin\{1,N\}$ segments (\emph{blocks}) permuted by $\pi$; call
$k$ the \emph{period} of this block structure. Let $p_1<\dots <p_s$ be
periods of all possible block structures on $\pi$. Call the finite
string $(p_1/1,$ $p_2/p_1,$ $\dots,$ $p_s/p_{s-1}, N/p_s)$ the
\emph{renormalization tower of $\pi$}. 
The same terminology can be used for \emph{patterns}, i.e., for families
of cycles of interval maps inducing the same (up to a flip) cyclic
permutation. A renormalization tower $\mathcal M$ \emph{forces} a
renormalization tower $\mathcal N$ if every continuous interval map
with a cycle of pattern with renormalization tower $\mathcal M$ must
have a cycle of pattern with renormalization tower $\mathcal N$. We
completely characterize the forcing relation among renormalization
towers. Take the following order among natural numbers: $ 4\gg 6\gg
3\gg \dots \gg 4n\gg 4n+2\gg 2n+1\gg\dots \gg 2\gg 1 $ understood in
the strict sense. We show that the forcing relation among
renormalization towers is given by the lexicographic extension of this
order. Moreover, for any tail $T$ of this order there exists an
interval map for which the set of renormalization towers of its cycles
equals $T$.
\end{abstract}

\maketitle

\section{Introduction and statement of the results}\label{s:intro}

From the standpoint of the theory of dynamical systems, the simplest
type of limit behavior of a point under iterates of a continuous map
is periodic. Thus, \emph{cycles} (called also \emph{periodic orbits})
play an important role in dynamics. The description
of possible sets of types of cycles of maps from a certain class is a
natural and appealing problem.

Since maps that are topologically conjugate are considered equivalent,
it is natural to declare two cycles equivalent if there exists a
homeomorphism of the space which sends one of them onto the other one.
On the interval this means that if two cycles induce the same cyclic
permutations or the cyclic permutations coinciding up to a flip then
these cycles should be viewed as equivalent. Classes of equivalence
are then called \emph{cyclic patterns} (since we consider \emph{only}
cyclic patterns and permutations, we will call them simply
\emph{patterns} and \emph{permutations} from now on). If an interval
map $f$ has a cycle that belongs to a specific pattern (equivalently,
induces a specific permutation), then one can say that the pattern
(the permutation) is \emph{exhibited} by $f$. Thus, one comes across a
problem of characterizing possible sets of patterns (or sets of
permutations) exhibited by interval maps. Here the permutations
induced by cycles on the interval (and considered up to the flip)
should be thought of as types of cycles.

Permutations define corresponding patterns. Vice versa, given a
pattern, there are two permutations (one of them is a flip of the
other one) that define the pattern. Permutations may have various
dynamics properties (e.g., one can talk about permutations of a
certain period, etc). In this case we will also say that the
corresponding pattern has these dynamic properties (e.g., one can talk
about patterns of a certain period, etc). In what follows we will
mostly talk about patterns.

How does one describe patterns? This naive question seems to have an
obvious answer: patterns should be described by permutations that they
are. However an important drawback of that approach is that such
description is too detailed and complicated. To have more information
may not always be better because then the structure of the set of all
patterns exhibited by a map is buried under piles of inessential
details. In other words, permutations chosen as types of cycles are
not necessarily a good choice because then the set of all permutations
induced by cycles of a given interval map may have a very complicated
structure not allowing for a transparent description.

A different (opposite in some sense) approach is to describe a pattern
by stripping it of all its characteristics but one: the period. Then,
of course, a lot of very different patterns will be lumped into one
big group of patterns of the same period. This approach may seem to be
too coarse and imprecise. However it is this idea that was adopted in
one-dimensional dynamics, thanks to a remarkable result, obtained by
A. N. Sharkovsky in the 1960s (see \cite{sha64} and~\cite{shatr} for
its English translation). To state it let us first introduce the {\it
  Sharkovsky ordering} for positive integers:
\[
3\sha 5\sha 7\sha\dots\sha 2\cdot3\sha 2\cdot5\sha 2
\cdot7\sha\dots\sha 4\sha 2\sha 1.
\]
Denote by $\sh(k)$ the set of all integers $m$ with $k\sha m$,
including $k$ itself, and by $\sh(2^\infty)$ the set
$\{1,2,4,8,\dots\}$; denote by $\Per(g)$ the set of periods of cycles
of a map $g$.

\begin{thmS}
If $g:[0,1]\to [0,1]$ is continuous, $m\sha n$ and $m\in \Per(g)$ then
$n\in \Per(g)$ and so there exists $k\in\mathbb N\cup 2^\infty$ with
$\Per(g)=\Sh(k)$. Conversely, if $k\in\mathbb N\cup 2^\infty$ then
there exists a continuous map $f:[0, 1]\to [0, 1]$ such that $\Per(f)
= \Sh(k)$.
\end{thmS}

The role of the Sharkovsky Theorem in the theory of dynamical systems
is, in particular, based upon the fact that its first part can be
understood in terms of \emph{forcing relation}. Indeed, it states that
if $m\sha n$ then the fact that an interval map has a point of period
$m$ \emph{forces} the presence of a point of period $n$ among the
periodic points of the map. In short, \emph{period $m$ forces period
  $n$}. Notice that with this understanding of the concept of forcing,
every number forces itself. If we think of the period of a cycle as
its type, we can view the Sharkovsky Theorem as a result showing how
such types of cycles (i.e., their periods) force each other.

It is natural to try to replace the period by some notion that
provides more information, but is easy to interpret and handle. To
explain our choice we discuss some properties of interval cycles
below. Any of the concepts that we introduce will be defined only for,
say, permutations, but can be similarly defined for the corresponding
cycles and patterns; the same applies to the corresponding notation.

\begin{definition}[Block structure]\label{d:block-str}
A cyclic permutation $\pi:\{1, \dots, N\}\to \{1, \dots, N\}$ has a
\emph{block structure} if there is a partition of $\{1, \dots, N\}$
into $k\notin\{1,N\}$ segments (\emph{blocks}) permuted by $\pi$. We
will call $k$ the \emph{period} of this block structure. A permutation
has a \emph{division} if it has block structure of period $2$.
\end{definition}

We will show later that if two block structures of the same
permutation have periods $p<q$ then $q$ is a multiple of $p$.

A permutation with a block structure can be studied in two steps:
study the factor-permutation obtained if each block is collapsed to a
point, while the order among blocks is kept, and then study the
restriction of the permutation to blocks. Evidently, this step-by-step
approach is easiest to implement if one uses the most basic, and,
therefore, the smallest steps. Thus, it is natural to consider all
possible block structures and navigate among them moving toward larger
and larger periods of blocks while making the smallest possible steps.

\begin{definition}[Renormalization towers]\label{d:decotow}
Let $p_1<\dots <p_s$ be periods of all possible block structures on
$\pi$. Call the finite string $(p_1/1,$ $p_2/p_1,$ $\dots,$
$p_s/p_{s-1}, N/p_s)$ the \emph{renormalization tower of $\pi$} and
denote it by $\dtp(\pi)$. Usually we will call a renormalization tower
just a \emph{tower}, or, if we want to distinguish it from an infinite
tower (see Definition~\ref{d:inf-tow}), a \emph{finite tower}. Call the
cardinality $s+1$ of the tower of $\pi$ its \emph{height}. For
consistency, if $\pi$ has no block structure, we define its tower as
$(N)$. While formally the tower of the permutation on $\{1\}$ should be
$(1)$ (so the number $1$ appears there, unlike in other towers) or
$\emptyset$, we will ignore this exception and pretend that the
singleton of a fixed point is not a cycle.
\end{definition}

\begin{figure}
\begin{center}
\includegraphics[width=120truemm]{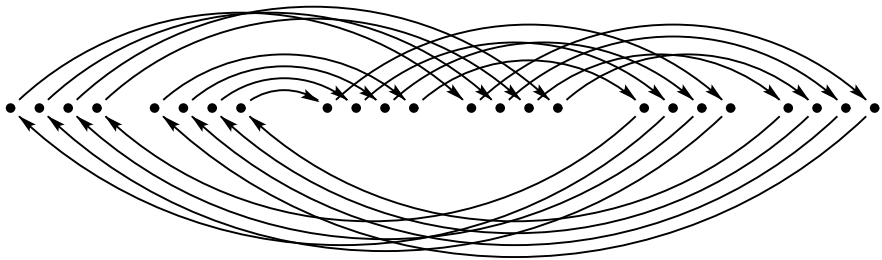}
\caption{A cycle of period 24 with tower $(3,2,4)$.}\label{period24}
\end{center}
\end{figure}

Observe that by definition no number in the tower of
$\pi$ equals $1$ (except the trivial case of $N=1$, see
Definition~\ref{d:decotow}). The concept of
the tower extends that of the period of a cyclic
permutation and reflects the structure of a cyclic permutation in a
more detailed way than the period itself. It combines both the
combinatorics and the order among points, but can be viewed as
\emph{numeric} in its nature. By this we mean that even though the
tower of a permutation can be a string a several
numbers, its height does not necessarily increase to infinity with
period. In fact, the tower of a permutation with no
block structure of \emph{any} period, however big, is $1$. This
distinguishes towers from permutations themselves.

The concept of renormalization plays a prominent role in smooth
dynamics. In a nutshell, it means that under some circumstances one
can take a specific region in the space, consider the first return map
on this region, normalize the size of the region to the original size
of the space, and discover that it is a map of the same type as the
original map. In the majority of cases the region in question is
itself periodic. This process can be sometimes repeated infinitely
many times, leading to regions of smaller and smaller sizes and higher
and higher periods. In that case in the end one gets so-called
\emph{infinitely renormalizable limit sets} (the corresponding maps
are also called \emph{infinitely renormalizable}), and metric
properties describing how smaller periodic regions relate to greater
periodic regions in terms of size, are of utter importance for one's
understanding of the smooth dynamical system in question. Our approach
is necessarily combinatorial and may be viewed as a discrete version
of the above described infinite renormalization process.

In this paper we consider towers as types of cycles,
and, according to the approach discussed earlier, aim at giving a full
description of sets of towers of cycles of continuous
interval maps. Similar to other cases (for example, to the case of
periods of cycles), in order to do so we have to study the problem of
the coexistence of towers. More precisely, we want to
characterize the forcing relation among them.

\begin{definition}[Tower forcing]\label{d:force-tow}
Let $\mathcal M$ and $\mathcal N$ be two finite strings of integers
each of which is greater than $1$. Suppose that every continuous
interval map that has a cycle with tower $\mathcal M$ has a cycle with
tower $\mathcal N$. Then we say that the
tower $\mathcal M$ \emph{forces} the
tower $\mathcal N$.
\end{definition}

Evidently, in the above definition, one can replace ``every continuous
interval map that has a cycle with tower $\mathcal M$
has a cycle with tower $\mathcal N$'' by ``every pattern
with tower $\mathcal M$ forces a pattern with tower
$\mathcal N$.''

Notice that it is not at all clear if the notion of forcing is
meaningful and provides for a transparent order among
towers. In this paper we solve this problem, show that the order
corresponding to the forcing relation among towers is linear, and
provide its full description.

Namely, in a recent paper \cite{bm18} the following order among natural
numbers was defined:
\begin{equation}
4\gg 6\gg 3\gg \dots \gg 4n\gg 4n+2\gg 2n+1\gg\dots \gg 2\gg 1
\end{equation}
Here we understand the order $\gg$ in the strict sense. In other
words, $k\gg k$ is not true for integers $k$. We will call this order
the \emph{nbs order} (the term comes from the acronym ``no block
structure'' and will be justified later on).

The nbs order covers all positive integers. All positive integers
except for $1$ and $2$ are ordered in a rather transparent fashion.
Namely, first we order all even numbers in the increasing way. Then we
insert odd numbers $2n+1$ between even numbers $4n+2$ and $4n+4$. At
the end we put the missing so far numbers $2$ and $1$.

\begin{definition}[$\gg$-tail]\label{d:gg-tail}
A set $A$ of numbers is said to be a \emph{$\gg$-tail} if for any
number $m\in A$ and any number $n$ with $m\gg n$ the set $A$ must
contain $n$.
\end{definition}

Clearly, the structure of a $\gg$-tail can be described explicitly.
Namely, given a $\gg$-tail $A$, choose the $\gg$-greatest number $n$
of all numbers in $A$. It is easy to see that the $\gg$-greatest
element of $A$ always exists. Hence $A$ is the set $N(n)$ of all
numbers $x$ such that $x=n$ or $n\gg x$.

\begin{definition}[Tower]\label{d:towers-abstr}
A finite string of integers larger than $1$ is said to be a \emph{tower}.
\end{definition}

The order we actually want to introduce for towers is the
\emph{lexicographic extension} of $\gg$ onto the family of towers. Let
\[
\mathcal N=(n_1,n_2 \dots, n_k),\ \ \ \mathcal M = (m_1,m_2 \dots,
m_l)
\]
be two towers (see~\cite{Tol}). Append each of them by infinite
strings of $1$s and denote these \emph{canonical extensions} by
$\mathcal N'$ and $\mathcal M'$. Observe that in any canonical
extension constructed above the number $1$ does not show in the
beginning for some time, but once it shows, the tail of the tower must
consist of $1$s only. Let $s$ be the first place at which $\mathcal
N'$ and $\mathcal M'$ are different. Then we write $\mathcal N\gg
\mathcal M$ if $n_s\gg m_s$. We keep the same notation $\gg$ for the
lexicographic extension of $\gg$, as no confusion arises. Evidently,
the family of all towers with $\gg$-order is linearly ordered,
as any two towers are comparable in the sense of $\gg$-order.

Before we state our main theorem, we need to introduce additional
notions (similar to $\sh(2^\infty)$ in the Sharkovsky ordering).

\begin{definition}[$\gg$-tail of towers]\label{d:ggl-tail}
A set $A$ of towers is said to be a \emph{$\gg$-tail} if for any tower
$\M\in A$ and any tower $\N$ with $\M\gg \N$ the set $A$ must contain
$\N$.
\end{definition}

It is easy to give a direct and explicit description of all possible
$\gg$-tails for towers (based exclusively upon properties of nbs
order). Namely, observe that all towers directly generated by cycles
have no digits $1$ in them (except $(1)$; this is an exception that we
will mainly ignore later).

\begin{definition}[Infinite tower]\label{d:inf-tow}
An \emph{infinite tower} is an infinite string $\N$ of positive
integers $(m_0,$ $m_1,$ $\dots)$ such that either (a) $m_i>1$ for
every $i$, or (b) there exists minimal $j$ such that $m_i>1$ for each
$i<j$ and $m_i=1$ for every $i\ge j$. In case (b) we identify the
infinite tower $\M=(m_0, \dots)$ with finite tower of all initial
numbers $(m_0, m_1, \dots, m_{j-1})$ of $\M$ not equal to $1$. The
relation $\gg$ extends onto all infinite towers lexicographically.
\end{definition}

In particular, when we write $\M\gg \N$ for an infinite tower $\M$ and
a finite tower $\N$, we mean that $\M\gg \N'$ in the sense of
Definition~\ref{d:inf-tow} (here the infinite tower $\N'$ is the
canonical extension of the tower $\N$).

\begin{definition}[$\Tow(\M)$]\label{d:tow-set}
Given an infinite tower $\M$, the set $\Tow(\M)$ is defined as the set
of all finite towers $\N$ such that $\M\gg \N$, together with the finite
tower $\widetilde \M$ in case $\M$ is the canonical extension of
$\widetilde \M$.
\end{definition}

Observe that infinite towers of type (1) from
Definition~\ref{d:inf-tow} play here similar role as $2^\infty$ for
the Sharkovsky ordering.

\begin{lemma}\label{towertails}
Every $\gg$-tail of towers is of the form $\Tow(\N)$ for some infinite
tower $\N$. Conversely, for every infinite tower $\N$ the set
$\Tow(\N)$ is a $\gg$-tail of towers.
\end{lemma}

\begin{proof}
The second part of the lemma is obvious, since the order $\gg$ is the
same in the set of finite and infinite towers. To prove the first
part, assume that $\mathcal A$ is a $\gg$-tail of towers. Then we
construct the sequence $(m_j)$ by induction. If $m_i$ for $i<j$ are
defined, then $m_j$ is the $\gg$-largest number such that there is a
tower in $\mathcal A$ starting with $(m_0,\dots,m_{j-1},m_j)$ (and $1$
if such number does not exist). Then for the infinite tower
$\N=(m_0,m_1,\dots)$ we have ${\mathcal A}=\Tow(N)$.
\end{proof}

Now we can state the main result of the paper.

\begin{thmM}
If $\mathcal N\gg \mathcal M$ and a continuous interval map $f$ has a
cycle with tower $\mathcal N$ then it has a cycle with tower $\mathcal
M$, and so there exists an infinite tower $\mathcal K$ such that the
set of all (finite) towers of cycles of $f$ is $\Tow(\mathcal K)$.
Conversely, if $\mathcal K$ is an infinite tower, then there exists a
continuous interval map $g$ such that the set of all (finite) towers
of cycles of $g$ coincides with $\Tow(\mathcal K)$.
\end{thmM}

The paper is organized as follows. In Section~\ref{s:prelim} we go
over preliminary information, including the tools related to the
forcing relation among patterns and results of the recent paper
\cite{bm18}. In Section~\ref{s:basic} we establish a few basic
facts concerning towers. In Section~\ref{s:unimodals}
we study unimodal cycles and patterns; they play a significant role in
the proof of the realization part of Main Theorem.
Section~\ref{s:main} is devoted to the proof of Main Theorem as well
as illustrating it with some applications.

\section{Preliminaries}\label{s:prelim}

Here we introduce basic notions and theorems that will be used in the
next parts of the paper.

\begin{definition}[Forcing relation among patterns]\label{d:patterns}
If $A$ and $B$ are patterns such that any continuous interval map with
a cycle of pattern $A$ has a cycle of pattern $B$, then one says that
$A$ \emph{forces} $B$. This is a partial order; we might talk of
forcing among cycles too.
\end{definition}

A nice description of patterns forced by a given pattern can be given
if we rely upon the notion of a \emph{$P$-linear map}. Observe that
one can talk of cycles on $\mathbb R$ or $I$ even if the map is not
defined outside the cycle.

\begin{definition}[$P$-linear maps and $P$-basic intervals]\label{d:plin}
Let $P\subset \mathbb R$ be a finite set. Set $I=[\min P,$ $\max P]$.
The closure of a component of $I\sm P$ is said to be a
\emph{$P$-basic} interval. If a map $f:P\to P$ is given, then the
\emph{$P$-linear map} $F:I\to I$ is defined as the continuous
extension of $f$, linear on each $P$-basic interval. Similarly, a map
$g:I\to I$ is said to be \emph{$P$-linear} if $g$ is monotone on each
$P$-basic interval.
\end{definition}

Mostly, one uses $P$-linear maps in dealing with cycles.

\begin{theorem}[\cite{alm}]\label{t:plin}
If $P$ is a cycle of pattern $A$ and $f$ is a $P$-linear map, then the
patterns of cycles of $f$ are exactly the patterns forced by $A$.
\end{theorem}

Properties of patterns can be stated in terms of $P$-linear maps. A
map $F:A\to A$ is called \emph{topologically exact} if for any
nonempty open set $U\subset A$ there is $n$ such that $F^n(U)=A$. We
will often write ``basic intervals'' meaning ``$P$-basic intervals''
if it is clear which cycle (or finite set) $P$ we mean.

\begin{proposition}\label{exact}
A cycle $P$ of period $n>2$ has no block structure if and only if the
$P$-linear map $f$ is topologically exact.
\end{proposition}

\begin{proof}
If $P$ has a block structure then any basic interval $I$ in the convex
hull of a block has images contained in the convex hulls of blocks, and
so $f$ is not topologically exact.

Now, assume that $P$ has no block structure. There must be a basic
interval $I=[a, b]$ such that $a$ and $b$ are mapped in opposite
directions. Then consecutive images of $I$ keep growing until some
image of $I$ covers $P$. Let $J$ be any basic interval and consider
the union of the intervals $f^k(J)$ over $k=0,1,2,\dots$. If it is not
connected then intersections of its components with $P$ form
non-trivial blocks, a contradiction. Hence the images of $J$
eventually cover $I$, and then the entire $P$.

It follows that if $x$ and $f^m(x)$ (for some $m>0$) both belong to the
interior of a basic interval $J$, then $|(f^m)'(x)|>1$ (we will refer
to this as an \emph{expanding property}). To see this, take the maximal
open interval $K\subset J$ containing $x$ such that $f^i(K)\cap
P=\emptyset$ for all $i=0,1,\dots,m$. Then $f^m|_J$ is monotone, and
for each endpoint $y$ of $K$ there has to be $i\le m$ such that
$f^i(y)\in P$, as
otherwise $K$ can be slightly enlarged while still fulfilling the
property that defines it. Hence $f^m(\overline{K})=J$, and
so if $|(f^m)'(x)|\le 1$ then actually $|(f^m)'(x)|=1$ and $\overline{K}=J$, a
contradiction with the fact that an eventual image of $J$ covers $P$.

Let now $K$ be an open interval. If there is no $k$ such that $f^k(K)$
contains a point of $P$, then there is a basic interval $J$ such that
the set $Z$ of those integers $i\ge 0$ for which $f^i(K)\subset J$ is
infinite. Let $V$ be the union of the intervals $f^i(K)$ over $i\in Z$.
By the expanding property, the lengths of those intervals $f^i(K)$ are
larger than or equal to the length of $K$. Therefore the set $L$ has
finitely many components. However, for the longest component, say $H$,
there is $m>0$ such that $f^m(H)\subset L$ and the length of $f^m(H)$
is larger than the length of $H$, a contradiction.
This shows that we may assume that $K$ contains a point of
$P$. By shortening $K$, we may assume that it is a short interval
containing a point of $P$ as its endpoint.

By taking $2n$ first images of such $K$ we get $2n$ short intervals
containing a point of $P$ as an endpoint, so one of them has to contain
some other one. By the expanding property, they have different lengths,
and the longer one is the image of the shorter one under $f^j$ for some
$j>0$. Itereating f$^j$ further, we get longer and longer intervals,
until some image of $K$ contains a basic interval. Now we see that some
further image of $K$ is equal to the convex hull of $P$, as desired.
\end{proof}

\begin{definition}[Loops of (basic) intervals]\label{d:loopsi}
Suppose that $f$ is an interval map and there are intervals $I_0,$
$\dots,$ $I_{n-1}$ such that $I_1\subset f(I_0),$ $I_2\subset f(I_1),$
$\dots,$ $I_0\subset f(I_{n-1})$. Then a finite string $L=I_0\to
I_1\to \dots \to I_{n-1}\to I_0$ is said to be an \emph{loop of
  intervals}. If intervals $I_0,$ $\dots,$ $I_{n-1}$ are $P$-basic for
some finite invariant set (e.g., cycle) $P$, then $L$ is said to be an
\emph{loop of ($P$-)basic intervals.}
\end{definition}

The next lemma helps one find various periodic points.

\begin{lemma}[see, e.g., \cite{block80}]\label{l:p-in-loops} If $L=I_0\to
I_1\to \dots \to I_{n-1}\to I_0$ is a loop of intervals, then there
exists a point $x$ such that $f^j(x)\in I_j$, $0\le j\le n-1$, while
$f^n(x)=x$.
\end{lemma}

The orbit of $x$ from Lemma~\ref{l:p-in-loops} is said to
\emph{correspond to the loop $L$.}

\begin{definition}[\v Stefan pattern]\label{d:stefan}
Consider the cyclic permutation $\sigma:\{1, 2,$ $\dots,$ $2n+1\}\to
\{1, 2, \dots, 2n+1\}$ defined as follows:
\begin{itemize}
\item $\sigma(1)=n+1$,
\item $\sigma(i)=2n+3-i$,\ \ if\ \ $2\le i\le n+1$,
\item $\sigma(i)=2n+2-i$,\ \ if\ \ $n+2\le i\le 2n+2$,
\end{itemize}
see Figure~\ref{Stefan}. Then the pattern of this cyclic permutation
is called the \emph{\v Stefan} pattern. The interval $[n+1, n+2]$ is
then called the \emph{central} interval of the cycle $\{1, \dots,
2n+1\}$. Moreover, any cycle $P$ of this pattern is said to be a
\emph{\v Stefan} cycle, and the $P$-basic interval corresponding to
the central interval of the pattern is also called the \emph{central}
interval of this cycle.
\end{definition}

\begin{figure}
\begin{center}
\includegraphics[width=120truemm]{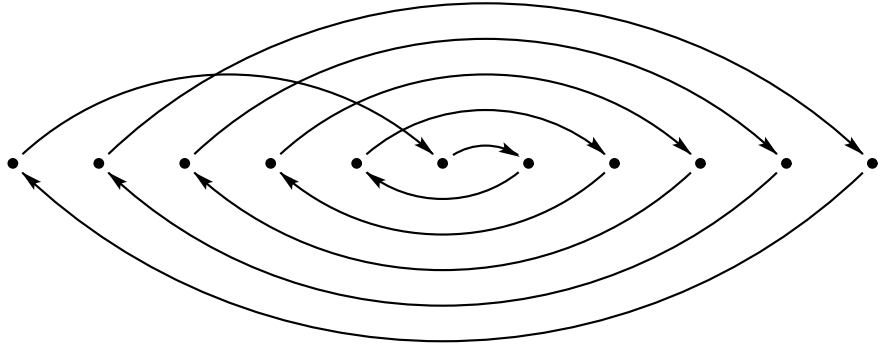}
\caption{A \v Stefan cycle of period 11.}\label{Stefan}
\end{center}
\end{figure}

Let $f$ be the $P$-linear map defined by the \v Stefan cyclic
permutation from Definition~\ref{d:stefan}. Then it has the following
properties. Take the central interval $I=[n+1, n+2]$ of $P$. Then
$f(I)=[n, n+2], f^2(I)=[n, n+3]$ etc. In other words, with each
application of $f$ the image of the central interval $I$ grows, adding
one new $P$-basic interval to the left or to the right of $I$,
alternately. In particular, $f^j(I)$ contains $j+2$ points of $P$.
This lasts until on the $2n-1$-st step we have $f^{2n-1}=[1, 2n+1]$.
The slow rate of growth of $I$ is ``responsible'' for the fact that \v
Stefan pattern is the forcing-weakest among all patterns of the same
(odd) period.

\begin{theorem}[\cite{ste77}, see also Lemma 2.16
    of~\cite{alm}]\label{t:stefan}
A periodic pattern of period $2n+1$ forces the \v Stefan pattern of
period $2n+1$. A periodic pattern of period $2n+1$ which is not \v
Stefan forces the \v Stefan pattern of period $2n-1$.
\end{theorem}

Finally, we would like to make a simple but useful remark concerning
the Shar\-kov\-sky Theorem. It claims that if a continuous interval
map $f$ has a cycle $P$ of period $n$ and $n\sha m$ then $f$ has a
cycle $Q$ of period $m$. In fact, one can add to this that $Q$ can be
chosen \emph{inside} the convex hull of $P$.

\subsection{Patterns with (no) block structure and the order among
  their periods}\label{ss:nos-order}

Let us begin by discussing simple facts concerning patterns $A$
\emph{with} block structure. It turns out that there are several
patterns trivially forced by $A$. In particular, using
Lemma~\ref{l:p-in-loops} based upon the techniques of loops of
intervals (see Definition~\ref{d:loopsi}) one can easily prove the
following well-known corollary.

\begin{corollary}\label{c:gen-by-block}
Let $P$ be a cycle of pattern $A$ and let $f$ be the $P$-linear map.
Let $X_1, \dots, X_l$ be blocks of some block structure in $P$. Then
there exists an $f$-cycle that has exactly one point in the convex
hull of each $X_i$. Moreover, the pattern $B$ of any such $f$-orbit is
well-defined and is forced by $A$.
\end{corollary}

We shall say that the pattern of the cycle whose existence is
established in Corollary~\ref{c:gen-by-block} is \emph{generated} by
the block structure with blocks $X_1, \dots, X_l$. We will also say
that the cycle $P$ (and the pattern $A$) from
Corollary~\ref{c:gen-by-block} have block structure \emph{over} the
cycles described in that corollary (or over the pattern $B$ of those
cycles).

An important particular case is presented in the next definition.

\begin{definition}[Doubling]\label{d:doubl}
Let $P$ is a cycle of pattern $A$ with block structure such that each
block consists of two points. Denote by $B$ the pattern generated by
this block structure. Then we say that $A$ is a \emph{doubling} of
$B$.
\end{definition}

Observe that given a pattern $B$ of period larger than 1, there are
several (more than one) patterns that are doublings of $B$. This
follows from the fact that the map on each block of $A$ from
Definition~\ref{d:doubl} may be either increasing or decreasing.

\begin{figure}
\begin{center}
\includegraphics[width=120truemm]{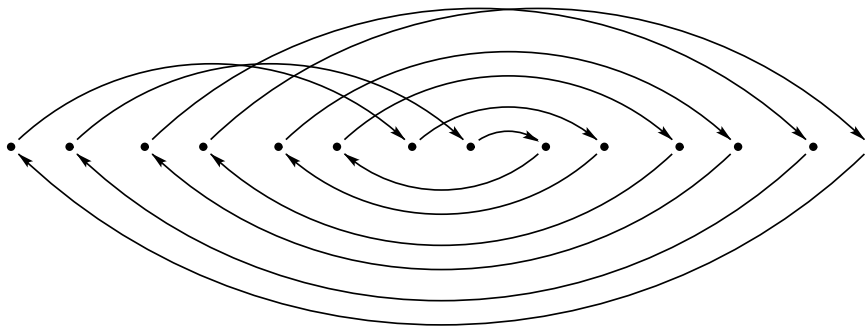}
\caption{A cycle of period 14, which is a doubling.}\label{period14d}
\end{center}
\end{figure}

\begin{figure}
\begin{center}
\includegraphics[width=120truemm]{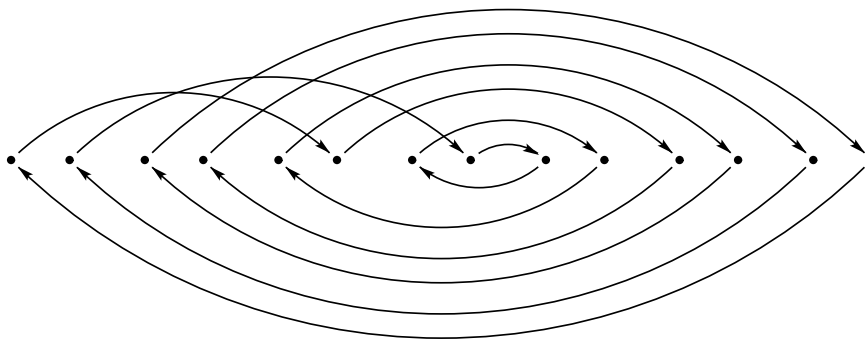}
\caption{A cycle of period 14, which is not a
  doubling.}\label{period14nd}
\end{center}
\end{figure}

The next lemma follows immediately from the fact that for a \v Stefan
cycle the image of the leftmost basic interval contains more than half
of the points of the cycle.

\begin{lemma}\label{l:stemix} \v Stefan cycles have no block structure.
\end{lemma}

Let us now describe the order among periods of the patterns with no
block structure induced by the forcing relation. Let $N(p)$ be the set
of all integers $s$ with $p\gg s$ and $p$ itself. Given an interval
map $f$, let $NBS(f)$ be the set of periods of all $f$-cycles with no
block structure.

\begin{theorem}[\cite{bm18}]\label{t:mix-ord}
Let $f$ be a continuous interval map. If $m\gg s$ and $f$ has a cycle
with no block structure of period $m$ then $f$ has also a cycle with
no block structure of period $s$. Moreover, $NBS(f)=N(p)$ for some
$p$, and for every $p$ there exists an interval map $f$ such that
$NBS(f)=N(p)$.
\end{theorem}

For completeness we prove a simple lemma describing some cases of
Theorem~\ref{t:mix-ord}.

\begin{lemma}\label{l:case1}
The following properties of a continuous interval map $f$ are
equivalent:
\begin{enumerate}
\item $NBS(f)=\{1\}$ or $NBS(f)=\{1,2\};$
\item the map $f$ has no points of odd period greater than $1$.
\end{enumerate}
\end{lemma}

\begin{proof}
Suppose first that $NBS(f)=\{1\}$ or $NBS(f)=\{1,2\}.$ Then $f$ cannot
have periodic points of odd periods greater than $1$. Indeed,
otherwise by Theorem~\ref{t:stefan} $f$ must have a \v Stefan cycle,
and by Lemma~\ref{l:stemix} every \v Stefan cycle has no block structure, a
contradiction with the assumption. On the other hand, suppose that $f$
has no points of odd period greater than $1$. Then, by Lemma~2.1.6
of~\cite{alm}, all its $f$-cycles of periods greater than $1$ have
division, and, hence, a block structure.
\end{proof}

\section{Basic facts about renormalization towers}\label{s:basic}

As we mentioned in the introduction, we chose towers as the
characteristic of cycles that we want to study, as a compromise
between periods and permutations. Let us discuss shortly the idea of
this choice.

What we want, is a notion that on one hand provides some idea on the
dynamical structure of a cycle, but on the other hand allows us to
study the emerging order without making things too complicated to
handle. Ideally, the order should be linear. In particular, this
notion should be of a ``numerical'' type, rather than
``combinatorial.'' It turns out that the towers satisfy all our
conditions. The order is linear (see Main Theorem). A tower is a
finite string of natural numbers. Some important information about the
dynamical structure of a cycle can be read from its tower (not all the
information, but probably as much as we can count on if we want to
keep the other properties of that notion).

In the rest of this section we prove a few statements concerning
towers, but not dealing with the nbs order. Let $A$ be
a pattern with tower $\dtp(A)=(m_1, m_2, \dots, m_k)$.
Call the blocks from the block structure of period $m_1$ \emph{blocks
  of the first level}. Similarly, the blocks of the next block
structure are said to be \emph{blocks of the second level}, and so on.
It is easy to see that the block structure of the $s$-th level is of
period $m_1\cdot m_2\cdots m_s$; in particular, the period of $A$ is
$m_1\cdots m_k$.

First we relate periods of two block structures of the same cycle.

\begin{lemma}\label{l:blo-div}
Let $P$ be a cycle of period $n$ of an interval map $f$. Suppose that
$P$ has two block structures of periods $p<q$. Then $p$ divides $q$.
\end{lemma}

\begin{proof}
Denote by $X$ and $Y$ the leftmost blocks of the block structures
of $P$ of periods $p$ and $q$, respectively. Then, clearly, $Y\subset
X$. By definition this implies that $f^s(Y)\cap X=\emptyset$ for every
integer $s$ that is not divisible by $p$. Since $f^q(Y)=Y\subset X$,
it follows that $q$ is divisible by $p$.
\end{proof}

We will also need the next lemma.

\begin{lemma}\label{l:bs-bs}
Let $P$ be a cycle of an interval map $f$. Let $X_1, \dots, X_k$ be
blocks of a block structure of $P$. Moreover, suppose that there
exists $i$ such that $X_i$ is a cycle of $f^k$ with no block
structure. Then $X_1, \dots, X_k$ are blocks of the last non-trivial
block structure of $P$.
\end{lemma}

\begin{proof}
Indeed, otherwise $X_i$ would have a block structure for $f^k$, which
it does not have by the assumption.
\end{proof}

The next lemma immediately follows from definitions and
Lemma~\ref{l:blo-div}.

\begin{lemma}\label{l:bloc-tow}
Suppose that a pattern $A$ has a block structure over a pattern $B$.
Then we have $\dtp(B)=(m_0, m_1, \dots, m_r)$, while
\[
\dtp(A)=(m_0,m_1,\dots,m_r,m_{r+1},\dots,m_k)
\]
for appropriate $r<k$. In particular, all patterns $B$ over which $A$
has a block structure, have towers $(m_0),(m_0, m_1),\dots,(m_0, m_1,
\dots, m_{k-1})$.
\end{lemma}

The following lemma is actually a simple particular case of Main
Theorem. Given a tower $\mathcal M=(m_0, \dots, m_k)$
of height $k$, we will call any tower $(m_0, \dots,
m_i), i\le k$, a \emph{beginning} of $\mathcal M$.

\begin{lemma}\label{l:triv}
If a tower $\mathcal N$ coincides with a beginning of
a tower $\mathcal M$, then any pattern $A$ with pattern $\M$ forces a
pattern $B$ with tower $\N$, such that $A$ has a block structure over
$B$. In particular, $\mathcal M$ forces $\mathcal N$.
\end{lemma}

\begin{proof}
Let $P$ be a cycle of the $P$-linear map $f$, where $P$ has pattern
$A$. Let $X_1, X_2, \dots$ be the blocks of $P$ at the level equal the
height of $\mathcal N$. By Corollary~\ref{c:gen-by-block} there is a
cycle of $f$ generated by the block structure formed by these blocks.
By definition and the assumptions of the lemma, this cycle has some
pattern $B$ with tower $\mathcal N$, and such that $A$ has a block
structure over $B$. By Theorem~\ref{t:plin} it follows that $\mathcal
M$ forces $\mathcal N$.
\end{proof}

\section{Unimodal cycles and patterns}\label{s:unimodals}

Here we will state several well known facts about unimodal cycles and
patterns. A cycle $P$ of a $P$-linear map $f$ is \emph{unimodal} if
$f$ is unimodal, or the period of $P$ is 2 or 1. In particular, all
cycles of a unimodal map are unimodal. Patterns of unimodal cycles
will be also called unimodal.

We will use the \emph{kneading theory} for unimodal maps. It comes in
several versions; here we will use the version from~\cite{ce}, adapted
to considering cycles. Suppose that $A$ is a unimodal pattern of
period $n>1$. When we consider its permutation $\sigma:\{1,\dots,n\}
\to\{1,\dots,n\}$, we may assume that $\sigma(n)=1$ (that is, the
corresponding unimodal map has a local maximum in the interior of the
interval). Then there is $k\in\{1,\dots,n\}$ such that $\sigma(k)=n$.
In this situation, $\sigma$ is increasing on $\{1,\dots,k\}$ and
decreasing on $\{k,\dots,n\}$. The \emph{kneading sequence} of $A$ is
a sequence $S=(S_1,\dots,S_{n-1})$ of symbols $L,R$ (left, right),
where $S_i=L$ if $\sigma^i(k)<k$ and $S_i=R$ if $\sigma^i(k)>k$. Since
$\sigma(k)=n$, then $S_1=R$; since $\sigma(n)=\sigma^2(k)=1$, then
$S_2=L$. A pattern $A$ is \emph{even} if the number of symbols $R$ in
$S$ is even, and \emph{odd} if this number is odd. It is easy to see
that different patterns have different kneading sequences.

The \emph{star product} $S*T$ of the kneading sequence $S$ of a
pattern $A$ and the kneading sequence $V=(V_1,\dots,V_{m-1})$ of a
pattern $B$ is defined as the concatenation
\begin{align*}
SV_1SV_2\dots,SV_{m-1}S\ \ &\textrm{if $A$ is even, and}\\
S\check{V}_1S\check{V}_2\dots,S\check{V}_{m-1}S\ \ &\textrm{if $A$ is
  odd,}
\end{align*}
where $\check{L}=R$ and $\check{R}=L$. It turns out that $S*T$ is the
kneading sequence of some pattern $C$. This pattern has a block
structure over $A$. If a cycle $P$ has pattern $C$ and $f$ is the
$P$-linear map, then each block is itself a cycle of $f^n$ (where $n$
is the period of $A$) and has pattern $B$. Since the map $f$ is
unimodal, it is monotone on each block except perhaps one. Thus, $C$
is a unimodal extension of $A$ with $B$ (see
Definition~\ref{d:exten}). We will write $C=A*B$.

Since there are unimodal patterns of all periods with no block
structure (for instance, with the kneading sequences $RL^{n-2}$),
applying the star product we can produce unimodal patterns with any
tower. Let us state it as a theorem.

\begin{theorem}\label{kneading1}
There are unimodal patterns with all towers.
\end{theorem}

Another approach to the unimodal patterns is to look at the cycles of
the full tent map with given patterns. The \emph{full tent map}
$T:[0,1]\to[0,1]$ is given by the formula $T(x)=2x$ if $x\le 1/2$ and
$T(x)=2-2x$ if $x\ge 1/2$. If $P$ is a cycle of $T$, let us denote by
$\alpha(P)$ the largest (rightmost) point of $P$. For every unimodal
pattern $A$ there is one (if $A$ is a doubling) or two (otherwise)
cycles of $T$ of that pattern (but if the period of $A$ is 1, we do
not treat $\{0\}$ as a cycle). If there is one cycle, say $P$, we will
write $\alpha(A)=\alpha(P)$. If there are two such cycles, say $P_1$
and $P_2$, we will write $\alpha(A)= \min(\alpha(P_1),\alpha(P_2))$.
It turns out that if $A$ and $B$ are unimodal patterns, then $A$
forces $B$ if and only if $\alpha(A)\ge \alpha(B)$. This property is
used, in particular, to provide a simple proof of the second
(``realization'') part of the Sharkovsky Theorem in~\cite{alm}.

There is an important point, $5/6$, which is a border between cycles
of $T$ with and without a division (observe that $T^2(5/6)=2/3$, the
fixed point of $T$). Namely, $\alpha(P)<5/6$ if and only if $P$ has a
division. The following theorem follows immediately from this fact.

\begin{theorem}\label{kneading2}
Any \v Stefan pattern forces all unimodal patterns with towers
starting from $2$ (that is, with a division).
\end{theorem}

\section{Main Theorem}\label{s:main}

Now we start taking into account the nbs order.

\begin{lemma}\label{l:divis-g}
Suppose that $n=ps$. Then one of the following holds:
\begin{enumerate}
\item[(a)] $p=1$,
\item[(b)] $p=2$,
\item[(c)] $n=4k+2$ and $p=2k+1$,
\item[(d)] $p\gg n$,
\item[(e)] $p=n$.
\end{enumerate}
\end{lemma}

\begin{proof}
If (a), (b) and (e) do not hold then $2<p<n$. Under this assumption,
we consider various cases. Observe that in the nbs order we have
\begin{equation}\label{eo}
3\gg 5\gg 7\gg 9\gg\dots\ \ \textrm{and}\ \ 4\gg 6\gg 8\gg 10\gg\dots.
\end{equation}

If $n$ is odd, then $p$ is an odd number smaller than $n$ but greater
than $2$, and by~\eqref{eo}~(d) holds. Similarly, if both $n$ and $p$
are even, then by~\eqref{eo}~(d) holds.

If $n=4k$ and $p$ is odd, we have $p\le 2k$, so $p\le 2k-1$. Now~(d)
follows from the fact that $2k-1\gg 4k$ and from~\eqref{eo}.

The remaining case is $n=4k+2$ and $p$ odd. If $p=2k+1$ then~(c)
holds. If $p\ne 2k+1$ then $p\le 2k-1$, and~(d) follows from the fact
that $2k-1\gg 4k+2$ and from~\eqref{eo}.
\end{proof}

Theorem~\ref{t:mix-ord} deals with patterns with no block structure.
Using Lemma \ref{l:divis-g} one can show that it also has consequences
related to patterns that admit block structure.

\begin{lemma}\label{c:bs-gg}
Suppose that $A$ is a pattern of period $n$ and no division. Then
either $n=4k+2$ and $A$ is a doubling of the \v Stefan pattern of
period $2k+1$, or $A$ forces a pattern of period $n$ with no block
structure. In particular, if $n\gg m$ and $A$ has no division then $A$
forces a pattern of period $m$ and no block structure.
\end{lemma}

\begin{proof}
If $A$ has no block structure, there is nothing to prove. Assume that
it has a block structure. Then the height $r$ of the tower
$\dtp(A)=(m_1, \dots, m_r)$ of $A$ is at least 2. Let $P$ be a cycle
of pattern $A$ of the $P$-linear map $f$. By
Corollary~\ref{c:gen-by-block}, the block structure of the first level
of $\dtp(A)$ generates a certain cycle $Q$ of pattern $B$ with no
block structure. Moreover, $B$ has period $m_1$ and is forced by $A$.

Suppose that $A$ does not force a pattern of period $n$ with no block
structure. If $m_1=2$, then $A$ has division, a contradiction. If
$m_1>2$, then by Lemma~\ref{l:divis-g} and Theorem~\ref{t:mix-ord},
since $B$ does not force a pattern of period $n$ and no block
structure, we have $n=4k+2$ and $m_1=2k+1$. Let us show that $B$ is
the \v Stefan pattern of period $2k+1$. Indeed, suppose that $Q$ is
not a \v Stefan cycle. Then by Theorem~\ref{t:stefan} $B$ forces the
\v Stefan pattern of period $2k-1$. By Theorem~\ref{t:mix-ord}, this
pattern forces a pattern of period $n=4k+2$ with no block structure, a
contradiction with the assumption. The last claim of the lemma
immediately follows.
\end{proof}

The first non-trivial step toward a proof of Main Theorem is to
consider two patterns with towers of equal height that differ at the
last level. While we could state it using the tower notation, it makes
sense to introduce a new notion. Namely, we will say that a pattern
$A$ has a \emph{direct block structure} over a pattern $D$, if $A$ has
a block structure over $D$, but has no block structure over any
pattern $B$ that has block structure over $D$. In other words, to go
from $D$ to $A$, we take into account only one additional level of the
associated tower.

\begin{lemma}\label{l:same-h}
Let a pattern $A$ of period $nm$ have a direct block structure over a
pattern $D$ of period $n$. Let $s>1$ be such that $m\gg s$. Then there
exists a pattern $B$ of period $ns$ with a direct block structure over
$D$, forced by $A$.
\end{lemma}

\begin{proof}
Consider a cycle $Q$ of pattern $A$ and the $Q$-linear map $g$. Let
$X_1, X_2, \dots, X_n$ be the blocks in $Q$ of the block structure
over $D$. These blocks are cyclically permuted by $g$, and each of
them consists of $m$ points. For each $i\in\{1,2,\dots,n\}$, the map
$g^n|_{X_i}$ defines some pattern of period $m$.

Suppose that all cycles of $g^n|_{X_i}$ have a division. Then
$X_i=X_i'\cup X''_i$ so that the convex hulls of $X'_i$ and $X''_i$
are disjoint, $g^n(X_i')=X''_i$ and $g^n(X''_i)=X'_i$. Let us call
sets $X'_i$, $X''_i$ \emph{half-blocks}. Thus, points $x,y\in P$
belong to the same half-block if and only if $y=g^{2kn}(x)$ for some
integer $k\ge 0$. Therefore, if $x,y\in Q$ belong to the same
half-block, then $g(y)=g^{2kn+1}(x)=g^{2kn}(g(x))$, so $g(x)$ and
$g(y)$ belong to the same half-block. This shows that $g$ maps
half-blocks to half-blocks, so the block structure of $A$ over $D$ is
not direct, a contradiction.

Hence, there exists $i, 1\le i\le n,$ such that $g^n|_{X_i}$ has no
division. Then, by Theorem~\ref{t:mix-ord} and by Lemma~\ref{c:bs-gg},
$g^n|_{[\min(X_i), \max(X_i)]}$ must have a cycle $S$ of period $s$
with no block structure. Let $P$ be the $g$-orbit of any point of $S$.
Then the pattern $B$ of $P$ has period $ns$ and is forced by $A$.
Moreover, by Lemma~\ref{l:bs-bs}, $P$ has a direct block structure
over $D$.
\end{proof}

Now suppose that we apply Lemma~\ref{l:same-h} with $s\ge 3$ odd.
There may be many patterns of period $ns$ with a direct block
structure over $D$, forced by $A$. However, at least one of them does
not force any other one. Let us investigate this forcing-minimal
pattern closer. Observe that studying forcing-minimal patterns from a
specific set of patterns gives us an explicit description of
the dynamics guaranteed for a map that has cycles with patterns from
the set of patterns in question.

\begin{lemma}\label{l:stefan-tow}
Assume that a pattern $B$ of period $ns$ has a direct block structure
over a pattern $D$ of period $n$. Moreover, suppose that $s\ge 3$ is
odd and that $B$ does not force any other pattern of the same period
that has a direct block structure over $D$. Let a cycle $P$ of the
$P$-linear map $f$ have pattern $B$. Denote by $X_i$ the blocks of its
block structure over $D$. Then $X_i$ is a \v Stefan cycle of $f^n$ for
every $i$.
\end{lemma}

\begin{proof}
Observe that the map $f$ cyclically permutes the convex hulls of
blocks $X_i$. Our assumptions imply that $X_i$ is a cycle of odd
period $s$ of $f^n$, whose pattern does not force any other pattern of
period $s$ with no block structure. Thus, $X_i$ is \v Stefan cycle of
$f^n$.
\end{proof}

On the other hand, observe that in our situation $f^n$ restricted to
the convex hull of $X_i$ is a priori not necessarily $X_i$-linear or
even $X_i$-monotone. Thus, we need some facts concerning arbitrary
continuous interval maps that have \v Stefan cycles. It will be
convenient to assume that their periods are strictly larger than 3 as
the case of period 3 has to be considered directly. Recall that if $P$ is
a \v Stefan cycle then by the central $P$-basic interval we mean the
$P$-basic interval such that the collections of $P$-basic intervals to
the left and to the right of it are of the same cardinalities (see
Definition~\ref{d:stefan}).

\begin{lemma}\label{l:center}
Suppose that $f$ is a continuous interval map with a \v Stefan cycle
$Q$ of period $s=2n+1>3$. Suppose that the central $Q$-basic interval
$I$ is such that $f^{s-3}(I)\supset Q$. Then $f$ has a point of an odd
period less than $s$ but larger than $1$ whose orbit is a \v Stefan
cycle. In particular this holds if $f(I)$ contains more than three
points of $Q$.
\end{lemma}

\begin{proof}
We may assume that $Q=\{q_1, \dots, q_{2n+1}\}$ and that the map
$f|_Q$ induces the permutation described in Definition~\ref{d:stefan}.
By the assumption, $[q_1, q_2]\subset f^{s-3}(I)$. On the other hand,
$I\subset f([q_1, q_2])$. By Lemma~\ref{l:p-in-loops} there exists a
point $z\in I$ such that $f^{s-3}(z)\in [q_1, q_2]$ while
$f^{s-2}(z)=z$. Clearly, $z$ is not $f$-fixed. Hence the period of $z$
is an (odd) divisor of $s-2$ which is greater than $1$. Together with
Theorem \ref{t:stefan} this proves the first claim.

To prove the second claim, notice that $f(I)\supset I$. In addition,
observe that $[q_2, q_s]\subset f^{s-3}(I)$. Thus, it suffices to show
that $q_1\in f^{s-3}(I)$. However, by the assumption $f(I)$ contains a
point $q_r$ with $r\ne n, n+1, n+2$. Now, it follows from the
properties of \v Stefan cycles that $f^{s-1}(q_n)=q_1,$
$f^{s-2}(q_{n+1})=q_1,$ and $f^{s-3}(q_{n+2})=q_1$. Hence for some
$j\le s-4$ we will have that $q_1=f^j(q_r)\in f^j(f(I))$ so that
$q_1\in f^{j+1}(I)\subset f^{s-3}(I)$ as desired.
\end{proof}

Let us turn our attention back to the cycle $P$ from
Lemma~\ref{l:stefan-tow} (we will use the notation from
Lemma~\ref{l:stefan-tow} too) and study it using Lemma~\ref{l:center}.
According to our definition, the pattern of $P$ is a forcing-minimal
pattern among all patterns of period $ns$ that have direct block
structure over a given pattern of period $n$ (here $s\ge 3$ is odd).
Recall that $X_i, 1\le i\le n,$ is a \v Stefan cycle of $f^n$. To
reflect the dynamics of \v Stefan cycles more closely we can
label points in $X_i$ differently than in Lemma~\ref{l:center} (the
notation from Lemma~\ref{l:center} reflects the relative location of
points of a cycle rather than the dynamics). Namely, we can denote
points of $X_i$ as $p_i^j$ in such a way that $f^n(p_i^j) = p_i^{j+1}$
(adding 1 is modulo $s$) and either
\[
p_i^s  < p_i^{s-2} < \ldots < p_i^3 < p_i^1 < p_i^2 < \ldots < p_i^{s-1},
\]
or
\[
p_i^{s-1} < p_i^{s-3} < \ldots < p_i^2 < p_i^1 < p_i^3 < \ldots < p_i^s.
\]

Throughout the following lemmas we shall keep this notation as well as
the assumption that $s$ is odd and $s> 3$. We will use notation
$\langle x,y\rangle$ for the closed interval with endpoints $x,y$
(here we do not assume that $x<y$). We will rely upon the standard
techniques based on Lemmas~\ref{l:p-in-loops} and \ref{l:center} and
follow the ideas from~\cite{alm}.

\begin{lemma}\label{I11.11}
If $1\le t \le n$ then $f^t(\{p_i^j : j = 3,4,\ldots,s-1\}) \cap
\langle f^t(p_i^1),f^t(p_i^2)\rangle = \emptyset$. In particular,
if $f^t(p_i^1)=p_{i+t}^j$, then either there are no elements of $P$ between
$p^j_{i+t}$ and $p^{j+1}_{i+t}$ or there is only one such element,
namely $f^t(p^s_i)$.
\end{lemma}

\begin{proof}
Assume the contrary. Then there exists $j \in \{3,4,\ldots,s-1\}$ such
that $f^t(p_i^j) \in \langle f^t(p_i^1),f^t(p_i^2) \rangle$. Hence
$f^n(p_i^j)=f^{n-t}(f^t(p_i^j))\in f^{n-t}(\langle
f^t(p_i^1),f^t(p_i^2) \rangle)\subset f^n(\langle p_i^1, p_i^2
\rangle)$ which shows that $f^n(\langle p_i^1, p_i^2 \rangle)$
contains not only points $p_i^1, p_i^2$ and $p_i^3$ but also the point
$f^n(p_i^j)=p_i^{j+1}$ distinct from any of them. By
Lemma~\ref{l:center} this implies that $f$ has a cycle of period $nk$
and a block structure over $D$ with $k\ge 3$ odd, a contradiction. The
last claim of the lemma follows now from the fact that $i+t$-th block
$X_{i+t}$ of $P$ consists of $f^t$-images of points of $X_i$.
\end{proof}

Lemma~\ref{I11.11} allows one to specify the location of points
$f^t(p_i^1), 1\le t \le n$.

\begin{lemma}\label{I11.12}
If $1\le t \le n$, then either $f^t(p_i^1)=p_{i+t}^1$ or
$f^t(p_i^1)=p_{i+t}^2$.
\end{lemma}

\begin{proof}
Let $f^t(p_i^1) = p_{i+t}^j$. Then $f^t(p_i^2) = p_{i+t}^{j+1}$
(recall the dynamical nature of our notation). By Lemma~\ref{I11.11}
either (a) there are no points of $P$ between $p^j_{i+t}$ and
$p^{j+1}_{i+t}$ or (b) there is only one such point, namely
$f^t(p^s_i)$. Let us also emphasize that
$f^n(p^j_{i+t})=p^{j+1}_{i+t}$, i.e. we are talking about a
point of $X_{i+t}$ and its $f^n$-image. Since $X_{i+t}$ is \v Stefan,
the properties of \v Stefan patterns imply that in case (a) we get $j
= 1$ as the only possibility. Similarly, in case (b) we get $j = 2$.
\end{proof}

The location of points $f^t(p_i^1), 1\le t \le n$ is further specified
in the next lemma.

\begin{lemma}\label{I11.13}
There is exactly one $t \in\{1,2,\dots,n\}$ such that $f(p_t^1) \neq
p_{t+1}^1$.
\end{lemma}

\begin{proof}
At least one $s$ such that $f(p_s^1) \neq p_{s+1}^1$ must exist
because $f^n(p_i^1) = p_i^2$ for all $i$. Suppose that there are
$t,r\in\{1,2,\dots,n\}$, $t \neq r$, such that $f(p_t^1) \neq
p_{t+1}^1$ and $f(p_r^1) \neq p_{r+1}^1$. We may assume that $r < t$
and that $t$ is the smallest of all the integers $i > r$ such that
$f(p_i^1) \neq p_{i+1}^1$.

By Lemma~\ref{I11.12}, we have $f(p_r^1)=p_{r+1}^2$ and
$f(p_t^1)=p_{t+1}^2$. Since $f^n(p_t^1)=f(p_t^2)$ and
$f^n(p_{t+1}^2)=p_{t+1}^3$, it follows that $f(p_t^2)=p_{t+1}^3$.
Moreover, if $r<i<t$, then $f(p_i^1)=p_{i+1}^1$, and so
$f(p_i^2)=p_{i+1}^2$. Therefore, $f^{t-r+1}(p_r^1) = p_{t+1}^3$. This
contradicts Lemma~\ref{I11.12}.
\end{proof}

Lemma \ref{I11.13} implies the following corollary.

\begin{corollary}\label{I11.14}
In our situation, there is exactly one $t \in\{1,2,\dots,n\}$ such
that $f|_{X_t}$ is not monotone.
\end{corollary}

The following terminology was introduced in \cite{alm}.

\begin{definition}[Extensions]\label{d:exten}
A pattern with a block structure over a pattern $A$, where the map is
monotone on all blocks, except perhaps one, is called an
\emph{extension} of $A$ (in~\cite{mini}, it is called a \emph{simple
  extension}). If the unique non-monotone restriction of the map on
the corresponding block is unimodal, the extension in question is said
to be \emph{unimodal}. If an extension of $A$ is such that the first
return map to a block belongs to a pattern $B$, then we say that this
is an \emph{extension of $A$ by a pattern $B$} (observe that if $A$ is
an extension of a pattern then the first return maps to blocks belong
to the same pattern). If $B$ is a \v Stefan pattern, this defines
\emph{\v Stefan extensions}.
\end{definition}

\begin{figure}
\begin{center}
\includegraphics[width=120truemm]{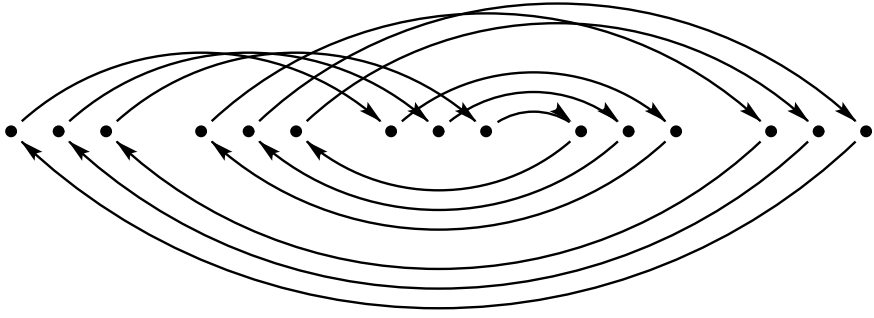}
\caption{A cycle of period 15, which is a unimodal extension of a \v
  Stefan cycle of period 5.}\label{period15}
\end{center}
\end{figure}

Lemma~\ref{l:3ext} is based on Corollary \ref{I11.14} and extra
arguments need in the case when $s=3$.

We say that $f^k$ has a \emph{horseshoe} if there are two closed
intervals $I, J$ with disjoint interiors, which have the property that
$f^k(I)\cap f^k(J)\supset I\cup J$.

\begin{lemma}\label{l:3ext}
A forcing minimal pattern $B$ among all patterns of period $ns$, where
$s\ge 3$ is an odd number, that have a direct block structure over a
given pattern $A$ of period $n$ must be a \v Stefan extension of
$A$.
\end{lemma}

\begin{proof}
By Corollary \ref{I11.14} it suffices to consider the case of $s=3$. Let $P$ be a
cycle of pattern $B$, and let $f$ be a $P$-linear map. Suppose that $B$ is
not an extension of $A$ and show that $f^n$ has a horseshoe on a block
of $P$ associated with the block structure of $B$ over $A$. Let $X$ be
a block of $P$ on which $f$ is not monotone and $r$ be such that
$f|_{f^r(X)}$ is not monotone. Let the convex hull of $X$ be $Z$, and
the convex hull of $f^r(X)$ be $T$. Then $Z=I\cup J$ is the union of
two $P$-basic intervals $I$ and $J$. Similarly, $T=K\cup L$ is the union of two
$P$-basic intervals $K$ and $L$. Since $f^r|_Z$ is not monotone, we may assume that
$f^r(I)=T$. Since $f^{n-r}|_T$ is not monotone, we may assume that
$f^{n-r}(K)=Z$, and so $f^n(K)=T$. If $K\subset f^r(J)$, then $I$ and
$J$ form a horseshoe of $f^n$. Otherwise $f^r(J)=L$. Since
$f^n(J)\supset I$, then $f^{n-r}(L)\supset I$ and $f^n(L)=T$. Hence
$K$ and $L$ form a horseshoe for $f^n$.
It is easy to see that the existence of a horseshoe for $f^n$
on the convex hull of a block of $P$ implies that $f$ has a cycle
$Q\ne P$ of period $3n$ whose pattern $D$ has a block structure
over $A$. Thus, $B$ forces $D$ and, therefore, is not forcing minimal,
a contradiction.
\end{proof}

The following result easily follows from Lemmas~\ref{l:same-h},
\ref{l:stefan-tow}, and~\ref{l:3ext}.

\begin{proposition}\label{ext}
Let a pattern $A$ of period $nm$ have a direct block structure over a
pattern $D$ of period $n$. Let $s>1$ be an odd integer such that $m\gg
s$. Then $A$ forces a pattern of period $ns$ which is an extension of
$D$ by a \v Stefan pattern.
\end{proposition}

In order to prove Main Theorem we also need two known results. We will
state them in a form consistent with out notation.

\begin{theorem}[Theorem~2.10.8 of~\cite{alm}]\label{2.10.8}
If a pattern $C$ is an extension of $B$ by $A$ and $A$ forces
$\widetilde{A}$ then $C$ forces an extension of $B$ by
$\widetilde{A}$.
\end{theorem}

The second one is a part of Theorem 9.12 from \cite{mini}.

\begin{theorem}\label{9.12}
Suppose that a pattern $A$ forces a pattern $B$ but does not have a
block structure over $B$, and $B$ is not a doubling. Then $A$ forces
extensions of $B$ by all unimodal patterns.
\end{theorem}

Now we are at last ready to prove the first claim of Main Theorem. We
state it as a separate theorem below. Recall that by
Lemma~\ref{towertails}, $\gg$-tails of towers coincide with the sets
$\Tow(\cdot)$ for infinite towers.

\begin{theorem}[Strong version of the first claim of Main
    Theorem]\label{t:1st}
Let $\mathcal N'$ and $\mathcal M'$ be two towers such that
\[
\mathcal N'=(m_1,\dots,m_t,m_{t+1},\dots,m_r),\ \ \ \mathcal
M'=(m_1,\dots,m_t, s),
\]
with $m_{t+1}\gg s$. Then any pattern $A$ of tower $\mathcal N'$
forces a pattern $C$ of tower $\mathcal M'$, and all unimodal
extensions of $C$. Thus, if $\mathcal N\gg \mathcal M$ and if a
continuous interval map $f$ has a cycle with tower $\mathcal N$ then
it has a cycle with tower $\mathcal M$. Hence, there exists an
infinite tower $\mathcal K$ such that the set of all finite towers of
cycles of $f$ is $\Tow(\mathcal K)$.
\end{theorem}

\begin{proof}
If a pattern $A$ has tower $\mathcal N'$, then by
Lemma \ref{l:triv} it forces a pattern $B$ with tower
$(m_1,$ $\dots,$ $m_t,$ $m_{t+1})$ over which $A$ has a block
structure ($A$ might also coincide with $B$). Then, by
Lemma~\ref{l:same-h}, $B$ forces a pattern $C$ with tower
$(m_1,\dots,m_t,s)$.

To prove the second claim, consider two cases. If $s>2$, then $C$
is not a doubling. Moreover, evidently the pattern $B$ does not have
block structure over the pattern $C$. Therefore, by
Theorem~\ref{9.12}, $B$ forces all unimodal extensions of $C$.

Let $s=2$ (and, hence, $m_{t+1}>2$). Denote by $D$ the pattern with
tower $(m_1, \dots, m_t)$ over which $B$ has a block
structure. Properties of the nbs order imply that there exists an odd
number $l>1$ such that $m_{t+1}\gg l\gg s_1=2$. By
Proposition~\ref{ext}, $B$ forces an extension of the pattern $D$ by
the \v Stefan pattern of period $l$. By Theorem~\ref{2.10.8}, $B$
forces extensions of $D$ by all patterns forced by the \v Stefan
pattern of period $l$. Therefore, by Theorem~\ref{kneading2}, $B$
forces all extensions of $D$ by a unimodal pattern with division as
desired.

The last part of the theorem immediately follows from the preceding
one.
\end{proof}

\begin{remark}
One can make a statement equivalent to Theorem~\ref{t:1st}, dealing
with patterns rather than cycles. Namely, if $\mathcal N\gg \mathcal
M$, then every pattern with tower $\mathcal N$ forces
a pattern with tower $\mathcal M$.
\end{remark}

To prove the second part of Main Theorem we will look at the family of
truncated tent maps $T_a$, $0<a\le 1$, given by $T_a(x)=\min(a,T(x))$,
where $T$ is the full tent map (see Section~\ref{s:unimodals}). This
is the same strategy that was used in order to prove the second part
of the Sharkovsky Theorem in~\cite{alm}. Recall that given a unimodal
pattern $D$ we define the number $\alpha(D)$ right after
Theorem~\ref{kneading1}.

\begin{theorem}[Second claim of Main Theorem]\label{t:2nd}
If $\mathcal K$ is an infinite tower, then there exists a continuous
interval map $g$ such that the set of all (finite) towers of cycles of
$g$ coincides with $\Tow(\mathcal K)$.
\end{theorem}

\begin{proof}
For every $n$ there is a unimodal pattern $B_n$ which is
forcing-minimal among unimodal patterns of period $n$ with no block
structure (they are identified in~\cite{mis94}). By
Theorem~\ref{2.10.8}, for a finite tower
$\M=(m_1,\dots,m_s)$, the unimodal pattern $A_\M=B_{m_1}*\dots
*B_{m_s}$ (see Section~\ref{s:unimodals}) is forcing-minimal among
unimodal patterns with tower $\M$. In other words,
$\alpha(A_\M)$ is smaller than $\alpha(C)$ for every other unimodal
pattern $C$ with tower $\M$. Thus, the map $T_\M$,
defined as $T_\M=T_{\alpha(A_\M)}$, has (by the first claim of Main
Theorem) cycles of all towers from $\Tow(\M)$ and (by
the connection between forcing and the values of $\alpha$) no cycles
of other towers, so that the function $\psi:\M\to
\alpha(A_\M)$ is monotone as a map from the space of all towers with
order $\gg$ to the interval $[0, 1]$.

Let $\K=(k_1,k_2,\dots)$ be an infinite tower with $k_n>1$ for every
$n$. For each $n$ we define a finite tower $\K_n$ by
$\K_n=(k_1,k_2,\dots,k_n)$. Then the sequence $\alpha(A_{\K_n})$ is
increasing, so it has the limit, which we will denote $\beta(\K)$.
Consider the map $T_\K=T_{\beta(\K)}$. We claim that it has cycles of
all towers from $\Tow(\K)$ and no cycles of other
towers.

The first part of the claim is immediate. Indeed,
$\beta(\K)>\alpha(A_{\K_n})$ for every $n$, so $T_\K$ has cycles of
all towers $\M$ for which there is $n$ with $\K_n\gg
\M$. However, these are exactly all towers from $\Tow(\K)$.

To prove the second part of the claim, consider a tower $\M=(m_1,
\dots, m_s)\notin\Tow(\K)$. Then $\M\gg \K_n$ for every $n$. Since
$\psi$ is monotone, then $\psi(\M)=\alpha(A_\M)>\psi(\K_n)$ which
implies that $\psi(\M)\ge \beta(\K)$. We claim that
$\psi(\M)>\beta(\K)$. We may assume
that $m_1=k_1,$ $\dots,$ $m_{j-1}=k_{j-1},$ but $m_j\gg k_j$. Let $D$
be the unimodal pattern with the kneading sequence $RL^3$. It is easy
to check that $\alpha(D)>\alpha(B_4)$, so by Theorem~\ref{t:mix-ord}
$D$ forces all patterns $B_n$. Therefore, by Theorem~\ref{t:1st},
$A_\M$ forces $B_{k_1}*\dots *B_{k_j}*D$, which in turn forces all
patterns $A_{K_n}$. Thus, $\alpha(B_{k_1}*\dots *B_{k_j}*D)$ is an
upper bound of the set $\{\alpha(A_{\K_n}):n=1,2,\dots\}$, so
$\alpha(C)>\alpha(B_{k_1}*\dots *B_{k_j}*D)$ cannot be its supremum.
This is a contradiction, and hence $T_\K$ has no cycles with towers
not in $\Tow(\K)$.
\end{proof}

\end{document}